\documentclass[12pt]{amsart}
\usepackage{color}
\usepackage[all]{xy}
\usepackage{amssymb}

\calclayout
\newtheorem{dummy}{anything}[section] 
\newtheorem{theorem}[dummy]{Theorem}
\newtheorem*{thma}{Theorem A}
\newtheorem*{thmb}{Theorem B}
\newtheorem*{thmc}{Theorem C}
\newtheorem{lemma}[dummy]{Lemma} 
\newtheorem{proposition}[dummy]{Proposition} 
\newtheorem{corollary}[dummy]{Corollary}
 
\theoremstyle{definition}
\newtheorem{definition}[dummy]{Definition}

 \newtheorem{remark}[dummy]{Remark}

\newcommand
{\eqncount}{\setcounter{equation}{\value{dummy}}%
\addtocounter{dummy}{1}}

\newcommand{\cN}{\mathcal N}
\newcommand{\cP}{\mathcal P}

\newcommand{\cI}{\mathcal I}

\newcommand{\bH}{\mathbf H}

\newcommand{\bP}{\mathbf P}
\newcommand{\CP}{\mathbf C \bP}
\newcommand{\CH}{\mathbf C\bH}

\newcommand{\bbL}{\mathbb L}

\newcommand{\bbZ}{\mathbb Z}

\newcommand{\cy}[1]{\bbZ/{#1}}
\newcommand{\wX}{\widetilde X}
\newcommand{\wM}{\widetilde M}
\newcommand{\bd}{\partial}
\newcommand{\vv}{\, | \,}

\newcommand{\La}{\Lambda}
\newcommand{\Zpi}{\bbZ \pi}

\newcommand{\mmatrix}[4]{\left (\vcenter
{\xymatrix@C-2pc@R-2pc{#1&#2\\#3&#4} }
\right )}

\newcommand{\la}{\langle}
\newcommand{\ra}{\rangle}
\DeclareMathOperator{\gdim}{g-dim}
\DeclareMathOperator{\Hom}{Hom}
\DeclareMathOperator{\wh}{Wh}
\DeclareMathOperator{\Mod}{mod}

\DeclareMathOperator{\Sharp}{\#}
\DeclareMathOperator{\Image}{Im}

\DeclareMathOperator{\Aut}{Aut}

\DeclareMathOperator{\sign}{sign}
\DeclareMathOperator{\Arf}{Arf}
\DeclareMathOperator{\cok}{cok}

\DeclareMathOperator*{\colim}{colim}

\newcommand{\Bpi}{K(\pi,1)}
 \newcommand{\sMdag}{(s_M)^{\dag}}
 \newcommand{\sMdagg}[1]{(s_{#1})^{\dag}}

\begin{document}

\title[$4$-manifolds with geometrically $2$-dimensional fundamental groups]
{Topological $4$-manifolds with geometrically $2$-dimensional fundamental groups}
\author{Ian Hambleton, Matthias Kreck, and Peter Teichner} 
\address{Department of Mathematics \& Statistics
 \newline\indent
McMaster University
 \newline\indent
Hamilton, ON  L8S 4K1, Canada}
\email{ian{@}math.mcmaster.ca}
\address{Hausdorff Research Institute for Mathematics
 \newline\indent
Poppelsdorfer Allee 82
 \newline\indent
D-53115 Bonn, Germany}
\email{kreck@HIM.uni-bonn.de}
\address{Department of Mathematics
 \newline\indent
University of California at Berkeley
 \newline\indent
 Berkeley, CA, 94720-3840, USA}
\email{teichner@math.berkeley.edu}
\thanks{The authors would like to thank the Max Planck Institut f\"ur Mathematik and the Hausdorff Institute in Bonn for its hospitality and support while this work was in progress. The authors were partially supported by NSERC Discovery Grant A4000,  DFG grant KR814, and NSF grant DMS-0453957.}

\date{Mar.~4, 2009}
\begin{abstract}\noindent
Closed oriented $4$-manifolds with the same geometrically $2$-dimensional fundamental group (satisfying certain properties) are classified up to $s$-cobordism by their $w_2$-type, equivariant intersection form and the Kirby-Siebenmann invariant. As an application, we obtain a complete homeomorphism classification of closed oriented $4$-manifolds with solvable Baumslag-Solitar fundamental groups, including a precise realization result.
\end{abstract}
\maketitle
\section{Introduction}\label{one}
In this paper we show how a combination of bordism theory and surgery can be used to classify certain closed oriented $4$-manifolds up to $s$-cobordism. Our results apply to  topological $4$-manifolds with geometrically $2$-dimensional fundamental groups, satisfying the three properties (W-AA) listed in Definition \ref{def:WAA} below.
We will prove in Section~\ref{sec:BS} that  these properties are satisfied by the family of {\em solvable Baumslag-Solitar groups} 
$$B(k) := \{ a,b\vv aba^{-1}=b^k\}, \quad k\in \bbZ.
 $$
The groups $B(k)$ have geometrical dimension $\leq 2$ because the 2-complex corresponding to the above presentation is aspherical. The easiest cases are
$$B(0) = \bbZ, \quad B(1) = \bbZ \times  \bbZ, \quad \text{and\ } B(-1) = \bbZ\rtimes \bbZ, $$
and these are the only Poincar\'e duality groups in this family. Each $B(k)$ is solvable, so is a ``good" fundamental group for topological $4$-manifolds \cite{freedman1}. This implies that Freedman's $s$-cobordism theorem is  available to complete the homeomorphism classification. This had been done previously only for the three special cases above, see \cite{freedman-quinn1} for $B(0)$, and \cite{hillman2} for $B(\pm 1)$, using a more classial surgery approach. 

A basic homotopy invariant of a $4$-manifold $M$ is the \emph{equivariant intersection form}, defined as the triple
$(\pi_1(M,x_0), \pi_2(M,x_0), s_M)$,
where $x_0\in M$ is a base-point, and
$$s_M \colon \pi_2(M,x_0) \otimes_{\bbZ} \pi_2(M,x_0) \to \bbZ[\pi_1(M,x_0)] 
$$ 
is the form described in Section \ref{sec:strategy}, formula (\ref{sMdef}). This pairing is $\Lambda$-hermitian, in the sense that for all $\lambda\in\Lambda:=\bbZ[\pi_1(M,x_0)] $ we have
\[ s_M(\lambda\cdot x, y) =\lambda \cdot s_M(x,y)  \quad \text{ and } \quad s_M(y, x)= \overline{s_M(x,y)}
\]
where $\lambda\mapsto \bar\lambda$ is the involution on $\Lambda$ given by the orientation character of $M$. In the oriented case studied below, this involution is determined by $\bar g = g^{-1}$ for $g \in \pi_1(M,x_0)$.

 An \emph{isometry} between two such triples is a pair $(\alpha, \beta)$, where $\alpha\colon \pi_1(M,x_0) \to \pi_1(M',x'_0)$ is an isomorphism of fundamental groups, and  $\beta\colon (\pi_2(M,x_0), s_M) \to (\pi_2(M',x'_0), s_{M'})$ is an $\alpha$-invariant isometry of the equivariant intersection forms. We will  assume throughout that our manifolds are connected, so that a change of base-points leads to isometric intersection forms. For this reason, we will omit the  base-points from the notation. 

\smallskip
Recall that an oriented $4$-manifold $M$ has type (I) if $w_2(\wM) \neq 0$, type (II) if $w_2(M)=0$, and type (III) if $w_2(M) \neq 0$ but $w_2(\wM) = 0$. The invariant  $w_2$ is defined for topological manifolds in (\ref{w2def}).

\begin{thma} For closed oriented $4$-manifolds with solvable Baumslag-Solitar fundamental groups, and given type and Kirby-Siebenmann invariant, 
any isometry between equivariant intersection forms can be realized by a homeomorphism.
\end{thma}

The invariants in Theorem A are not independent. For example, $M$ has type~(I)  if and only if the equivariant intersection form $s_M$ is \emph{odd}. This algebraic condition means that the identity component of $s_M(x,x)$ is an odd integer for some $x\in \pi_2(M)$. In the other case, if the identity component of $s_M(x,x)$ is always even, then we say that  $s_M$ is an {\em even} equivariant intersection form. This occurs in types (II) and (III).

\smallskip
The Kirby-Siebenmann invariant $KS(M) \in \cy 2$ (see \cite[p.~300]{kirby-siebenmann1}) is determined for spin manifolds by Rochlin's formula
$$ KS(M)\equiv \sign(M)/8 \ (\Mod 2)
$$
where $\sign(M)$ is the {\em signature} of the $4$-manifold $M$, defined via the ordinary intersection form on $H_2(M; \bbZ)$. For fundamental groups $\pi$ with $H_4(\pi;\bbZ)=0$ we show in Remark \ref{signature_remark} that this signature is determined by $s_M$ via the formula 
\[
\sign(M) = \sign(s_M\otimes_{\Lambda} \bbZ).
\]
 This formula does not hold in general, as one can see from examples of surface bundles over surfaces with nontrivial signature (but vanishing $\pi_2$). 

For $\pi_1(M)=B(k)$, type~(III) can only occur if $k$ is odd. In this case, we have the following generalization of Rochlin's formula, proven in Corollary~\ref{cor:Arf}:
 \[
KS(M) \equiv \sign(M)/8 + \Arf (M)\ (\Mod 2)
\] 
where $\Arf(M)\in \cy 2$ is a codimension~2 Arf invariant explained in Section~\ref{sec:BS}. 

For any closed oriented $4$-manifold $M$, the radical $R(s_M) \subseteq \pi_2(M)$ of the intersection form $s_M$ is completely determined by the fundamental group $\pi$ of $M$,  via an isomorphism $R(s_M)\cong H^2(\pi; \Zpi)$ proved in Corollary \ref{cor:rad}. We define $\pi_2(M)^\dag:=\pi_2(M)/R(s_M)$.

\begin{thmb}  For any closed oriented $4$-manifold $M$ with fundamental group $\pi=B(k)$, the quotient  $\pi_2(M)^\dag$ is a finitely generated, stably-free $\Zpi$-module and the induced form $\sMdag$ is non-singular. Conversely,
\begin{enumerate}
\item
Any non-singular $\Zpi$-hermitian form on a finitely generated, stably-free $\Zpi$-module is realized as $\sMdag$ by a closed oriented $4$-manifold $M$ with fundamental group $B(k)$. 
\item
Up to homeomorphism, there are exactly two such manifolds for odd forms, distinguished by the Kirby-Siebenmann invariant. If $k$ is even, an even form determines a manifold of type~\textup{(II)} uniquely; type~$\textup{(III)}$ does not occur. For $k$ odd, there is exactly one $4$-manifold with a given even intersection form in each type $ \textup{(II)}$ or $ \textup{(III)}$.
\end{enumerate}
\end{thmb}
\begin{remark} The classification of such stably free modules and  non-singular hermitian forms is a difficult algebraic problem. For example,
there exist odd intersection forms which are not extended from $\bbZ$ for manifolds with these fundamental groups (see \cite[\S 6]{fhmt1}). Moreover, we do not know the minimal rank of an intersection form with nontrivial Arf invariant.
\qed
\end{remark}
A finitely-presented group $\pi$ is \emph{geometrically $2$-dimensional} ($\gdim \pi \leq 2$) if there exists a finite aspherical $2$-complex with fundamental group $\pi$.  Examples of geometrically $2$-dimensional groups include free groups,  $1$-relator groups (e.g.~surface groups) and small cancellation groups \cite{lyndon-schupp1}, provided they are torsion-free, as well as many word-hyperbolic groups  \cite[2.3]{hillman1}, \cite[\S 10]{hillman4}.

 It is attractive to see whether the results of Theorems~A and B also hold for more general fundamental groups of geometric dimension $\leq 2$. In a series of papers \cite{hillman1}, \cite{hillman2}, \cite{hillman3}, \cite{hillman4}, Jonathan Hillman has investigated the homotopy classification of Poincar\' e $4$-complexes under various fundamental group assumptions. In the case of $\gdim\leq 2$, the problem was reduced to the minimal case, also studied here in Section~\ref{sec:minimal}. However, the homotopy classification of minimal models was not completed except for free or surface fundamental groups. Our focus is on the classification of $4$-manifolds up to $s$-cobordism, and our methods differ from the classical surgery approach in the sense that we do not need to understand the homotopy classification first.
 
\smallskip
 We now list the additional properties we will need for the fundamental groups (see \cite{taylor-williams1} for the surgery assembly maps).
\begin{definition}\label{def:WAA}
A group $\pi$ satisfies properties (W-AA) whenever
\begin{enumerate}
\item The Whitehead group $\wh(\pi)$ vanishes, 
\item The assembly map $A_5\colon H_5(\pi;\bbL_0) \to L_5(\Zpi)$ is surjective. 
\item The assembly map $A_4\colon H_4(\pi;\bbL_0) \to L_4(\Zpi)$ is injective. 
\end{enumerate}
\end{definition}

For a $4$-manifold $M$  with fundamental group $\pi$ and even equivariant intersection form $s_M$, there exists a unique class $w_M\in H^2(\pi;\cy 2)$ such that $w_2(M) = c^*(w_M)$ for a 2-equivalence $c\colon M\to K(\pi,1)$ that induces the identity on fundamental groups. In this setting, we define the  \emph{$w_2$-type} of $M$ to be the pair $(\pi,w_M)$. If $s_M$ is odd,  $M$ has type~(I) and we define the $w_2$-type to be the pair $(\pi_1(M),\textup{(I)})$.

 An \emph{isomorphism} between the $w_2$-types of $M$ and $N$ is an isomorphism $\pi_1(M) \cong \pi_1(N)$ that carries $w_M$ to $w_N$. In the odd case, the condition on $w_M$ is interpreted to mean that both $M$ and $N$ have type~(I). The $w_2$-type is a refinement of the type as soon as the action of $\Aut(\pi)$ on $H^2(\pi;\cy 2)$ has more than two orbits. The zero element determines a preferred orbit, corresponding to type~(II) above. The remaining orbits refine type~(III).
 
\begin{thmc} For closed oriented $4$-manifolds with geometrically $2$-dimensional fundamental groups satisfying properties \textup{(W-AA)}, and given Kirby-Siebenmann invariant, any isometry between equivariant intersection forms inducing an isomorphism of $w_2$-types can be realized by an $s$-cobordism. 
 \end{thmc}

\begin{remark}
It is  important to point out that whenever the Farrell-Jones isomorphism conjectures \cite{farrell-jones1} in algebraic $K$-theory and $L$-theory hold for a group $\pi$ with $\gdim\pi\leq 2$, then $\pi$ satisfies (W-AA), and in fact the assembly maps are isomorphisms. We refer to \cite{lueck-reich1} for a survey of results on these conjectures.
\end{remark}

\begin{remark}
A central tool in our classification is the {\em reduced intersection form} $\sMdag$ on the quotient  $\pi_2(M)^\dag:= \pi_2(M)/R(s_M)$ by the radical of $s_M$. This is a  finitely generated, stably-free $\Zpi$-module and $\sMdag$ is non-singular,  by Corollary \ref{cor:free}. Since  $s_{M_1}\cong s_{M_2}$ if and only if  $\sMdagg{M_1}\cong \sMdagg{M_2}$,   it follows that Theorems~A and C could be formulated with the reduced intersection form $\sMdag$ replacing the intersection form $s_M$  in the statements.
\end{remark}

Section~\ref{sec:strategy}  contains a summary of our classification technique as guide for the paper, and the proof of Theorem~C, modulo the computation of certain bordism groups.  Section~\ref{sec:minimal} gives some basic constructions and facts concerning $4$-manifolds, and Section~\ref{sec:stable} discusses stable classification of $4$-manifolds. We prove the first statement of Theorem~B in Corollary~\ref{cor:rad} and Corollary~\ref{cor:free},  just using stable classification of $4$-manifolds with geometrically 2-dimensional fundamental group. In Section~\ref{sec:Bbordism} we complete the bordism calculations and the proof of Theorem C.  In Section~\ref{sec:BS} we show that the solvable Baumslag-Solitar groups satisfy conditions (W-AA), deduce Theorem A from Theorem C, and prove the remaining parts of Theorem~B.

\medskip
\noindent {\bf Acknowledgement}: We wish to thank Wolfgang L\"uck for helpful conversations.

\section{The strategy: Reduced normal $2$-types and surgery} \label{sec:strategy}

In this section we will explain our strategy and reduce the proof of Theorems~A and C to the computation of  certain bordism groups, which we will analyze in Section~\ref{sec:Bbordism}. 
The classical surgery approach to classifying manifolds \cite{wall-book} would need as an input a homotopy equivalence  $h\colon N\to M$. Then one would ask whether $h$ is normally cobordant, over $M$, to a self-homotopy equivalence of $M$, and finally one would try to do surgery on such a normal cobordism until it becomes an $s$-cobordism. This last step can always be completed if the fundamental group $\pi$ in question satisfies the following subset of the properties (W-AA) in Definition~\ref{def:WAA}
\begin{equation} \tag{W-A}
\wh(\pi) = 0 \text{ and the assembly map } A_5\colon H_5(\pi;\bbL_0) \to L_5(\Zpi) \text{ is surjective.}
\end{equation}
Note that the vanishing of the Whitehead group $\wh(\pi)$ implies that one can suppress the decorations in the $L$-groups above. We'll refer to these properties of $\pi$ as `properties (W-A)' in the following. 

The other steps in this classical surgery approach depend on more than the fundamental group, for example the normal cobordism computation involves all homology groups of $M$.
In many instances, even the homotopy classification is too hard to understand.

In \cite{kreck3}, the second author developed a modified surgery approach to classification. In the case of $4$-manifolds, he starts with the {\em normal $2$-type} $B\to BSTOP$ of $M$. This is the second stage of the Moore-Postnikov factorization of  the classifying map $\nu_M\colon M \to BSTOP$ of the normal bundle (see \cite{kister2}, \cite{milnor3}, and \cite{rourke-sanderson1} for the construction of this classifying map). The map from $BSO$ to $BSTOP$ is a $3$-equivalence 
\cite[p.~300]{kirby-siebenmann1}, so $H^2(BSTOP;\mathbb Z/2) \cong \mathbb Z/2$. We call the non-trivial element $w_2$ and define 
\eqncount
\begin{equation}\label{w2def}
w_2(M) := \nu_M^*(w_2) \in H^2(M;\mathbb Z/2)
\end{equation}
The Moore-Postnikov factorization is a 3-coconnected fibration admitting a lift $\tilde \nu_M$ of $\nu_M$ which is a $3-$equivalence:
$$\xymatrix@!C
{& B\ar[d] \\ 
M\ar@{-->}[ur]^{\tilde\nu_M}\ar[r]^{\nu_M} & BSTOP
}$$
So the homotopy groups of the fibre of $B\to BSTOP$ vanish in degrees~$\ge 3$ and the {\em normal 2-smoothing} $\tilde\nu_M$ induces an isomorphism of homotopy groups in degrees~$\leq 2$. 

Given any fibration $B\to BSTOP$, one can define the normal  $B$-bordism groups $\Omega_4(B)$ of triangles as above \cite{stong1}, except that one does not require any homotopy theoretic conditions on the lift of $\nu_M$ to $B$. 

\begin{theorem}[\cite{kreck3}] If two closed $4$-manifolds admit  $B$-bordant normal 2-smoothings in the same $3$-coconnected fibration $B\to BSTOP$ then they are s-cobordant, provided their fundamental group satisfies properties~\textup{(W-A)} above.
\end{theorem}
\begin{remark}\label{rem:iso}
Note that two 4-manifolds that admit normal 2-smoothings in the same 3-coconnected fibration $B\to BSTOP$ must have isomorphic homotopy groups in degrees~$\leq 2$. More precisely, given normal 2-smoothings $\tilde\nu_M$, $\tilde\nu_N$, the induced maps on $\pi_1$ and $\pi_2$ are isomorphisms and by composing one such map with the inverse of the other one obtains particular isomorphisms $\pi_i(M) \cong \pi_i(N)$ for $i=1,2$. The main idea of \cite{kreck3} was that Poincar\'e duality would also force isomorphisms for $i>2$. This follows from showing that $M$ and $N$ are  $s$-cobordant, and hence simple homotopy equivalent.
\end{remark}
Even though the above theorem avoids an \emph{a priori} homotopy classification, the computation of the $B$-bordism groups can still be formidable. During the current collaboration, we realized that one can use the results of \cite{kreck3} to reduce the $s$-cobordism classification to an easier bordism question.

We will first recall the definition of the equivariant intersection form $s_M$ mentioned in the Introduction.
Let $M$ be an oriented $4$-manifold with universal covering $\widetilde M$. By Poincar\'e duality there is an isomorphism
$H^i_c(\widetilde M;\mathbb Z) \cong H_{4-i}(\widetilde M;\mathbb Z)$, from cohomology with compact support to homology. In particular, we obtain an isomorphism $$\varepsilon_0\colon H^4_c(\widetilde M;\mathbb Z) \to H_0(\widetilde M;\mathbb Z) =\mathbb Z.$$
 We identify $\pi_2(M)$ with $H_2(\widetilde M;\mathbb Z)$, by the Hurewicz Theorem, and then via Poincar\'e duality with $H^2_c(\widetilde M;\mathbb Z$). The cup product on $H^2_c(\widetilde M;\mathbb Z)$ is a $\pi_1(M)$-invariant bilinear form with values in $\mathbb Z$, and thus we obtain a $\Lambda$-hermitian form 
\eqncount
\begin{equation}\label{sMdef}
s_M(x,y) := \sum _{g\in G} \varepsilon_0(g^{-1}\tilde x \cup \tilde y)\cdot g\  \in \ \Zpi = \La,
\end{equation}
where $\tilde x,\tilde y\in H^2_c(\widetilde M;\mathbb Z)$ are the images of $x,y\in \pi_2(M)$ under the isomorphisms above.

\begin{definition}\label{def:reduced normal}
The {\em reduced normal $2$-type} of a $4$-manifold $M$ is a 3-coconnected fibration $B\to BSTOP$, admitting a lift of the classifying map $\nu_M$ of the normal bundle,  which is an isomorphism on $\pi_1$ and on $\pi_2$  is surjective with kernel the radical $R(s_M)$ of the intersection form $s_M$.  We call such a lift a {\em reduced normal $2$-smoothing} and note that it induces the epimorphism $\pi_2(M) \to \pi_2(B) \cong \pi_2(M)^\dag = \pi_2(M)/R(s_M)$. 
\end{definition}

\begin{theorem} \label{thm:reduced}
If two closed $4$-manifolds admit $B$-bordant {\em reduced} normal 2-smoothings in the same 3-coconnected fibration $B\to BSTOP$ then they are s-cobordant, provided their fundamental group satisfies properties~\textup{(W-A)} above. More precisely, the $s$-cobordism induces the same isomorphisms on $\pi_1$ and $(\pi_2)^\dag$ as the given reduced normal 2-smoothings.
\end{theorem}

\begin{proof}
Let $W$ be a normal $B$-bordism  between two reduced normal $2$-smoothings. We want to replace $W$ by an $s$-cobordism. Since the Whitehead group $\wh(\pi)$ is trivial we may ignore bases and look for an $h$-cobordism instead. We are now in the situation studied in \cite[Thm.4]{kreck3}.
 We may replace $W$ by surgeries below the middle dimension by a $2$-equivalence $c\colon W \to B$. Then the surgery obstruction $\Theta (W,c)$ in the obstruction monoid $l_5(\bbZ\pi)$ is defined. This obstruction is given by a half rank direct summand $V$ in a direct sum of hyperbolic planes over the group ring  $\Zpi$. There is an isometry from $V$ to the kernel of the map $\pi_2(M_1) \to \pi_2(B)$ \cite[Prop.8, i)]{kreck3}, which in our situation  is equal to the radical $R(s_M)$. Thus the quadratic form vanishes on $V$ and therefore $\Theta (W,c)$ is contained in the classical surgery group $L_5(\Zpi)$. Using our surjectivity assumption on the assembly map $A_5$,  we will show below that this element can be assumed zero in $L_5(\Zpi)$, by connected sum of the original $B$-bordism $W$ with a closed $5$-manifold, equipped with a suitable reference maps to $B$. Thus we are finished by applying \cite[Thm.4]{kreck3}.
 
The last step in the above argument can be explained in more detail, as follows.
The assembly map is a map
$$
\cN (M\times I,  M \times \{0,1\}) \to L_5(\Zpi),
$$
where $\cN (M\times I,  M \times \{0,1\})$ is the set of degree 1 normal maps
$
(T,H, \alpha)
$
where $\partial T = M + (-M)$, $H\colon T \to M\times I$ is a degree 1 map which is the identity on both boundary components, and $\alpha$ is a  stable framing of $\nu_T - H^*\nu_M$. Given $c\colon W \to B$ as in the proof above, we can form the sum of $W$ with $(T,H, \alpha)$ by glueing along a boundary component $M$. The new reference map is obtained by composition of $H$ with $c$. Up to bordism this sum  is just the connected sum of $W$  with the closed manifold obtained by identifying the two boundary components of $T$.

Finally we note that if we glue $T$ to $W$ along $M$, the surgery obstruction is the sum of the two surgery obstructions of $W$ and $T$. This is an immediate consequence of the construction of the surgery obstruction (see \cite[\S 6]{wall-book}).
Thus,  if we can realize every element of $L_5(\Zpi)$ by an appropriate $(T,H,\alpha)$, then by gluing $T$ to $W$ we can ensure  that the surgery obstruction vanishes. 
\end{proof}

To apply the above theorem, we need to understand when two $4$-manifolds admit reduced normal 2-smoothings into the same 3-coconnected fibration $B\to BSTOP$. For this purpose, we first describe $B$ in terms of a given $4$-manifold $M$. 

Let $P = P(M)$ denote a two-stage Postnikov system with  $\pi_1(P) = \pi_1(M)=\pi$, and  
$$
\pi_2(P) = \pi_2(M)^\dag = \pi_2(M) / R(s_M).
$$
Up to homotopy, $P$ is determined by these two homotopy groups and the $k$-invariant in $H^3(\pi; \pi_2(P))$. If $\gdim(\pi)\le 2$ this $k$-invariant is automatically zero. Thus $P$ is determined by $\pi_1(M)$ and $\pi_2(M) ^\dag$ alone in this case.

\begin{definition}\label{def:reduced} A map $c\colon M \to P$ is called a \emph{reduced  $3$-equivalence} if $c$ induces an isomorphism $\pi_1(M) \cong \pi_1(P)$, and an isomorphism $\pi_2(M)^\dag \cong \pi_2(P)$.
\end{definition}
Notice that a  reduced  $3$-equivalence is always a $2$-equivalence and it is a $3$-equivalence if and only if $R(s_M) = 0$. We will show in  Corollary~\ref{cor:rad}, that the radical $R(s_M) \cong H^2(\pi;\Zpi)$ only depends on $\pi$ and for most groups currently under consideration, this module is non-trivial.

\begin{lemma} \label{lem:w}
Let $c\colon M \to P$ be a  reduced  $3$-equivalence. Then there is a unique class $w\in H^2(P;\cy 2)$ such that $c^*(w)= w_2(M)$.  
\end{lemma}
\begin{proof}  We have a commutative diagram
$$\xymatrix{0 \ar[r] & H^2(\pi;\cy 2)\ar[r]\ar@{=}[d]& H^2(P;\cy 2) \ar[r]\ar[d]^{c^*} & H^2(\widetilde P, \cy 2)^\pi\ar[d]^{\tilde c^*}\ar[r] &H^3(\pi;\cy 2) \ar@{=}[d]\\
0 \ar[r] & H^2(\pi;\cy 2)\ar[r]& H^2(M;\cy 2) \ar[r]& H^2(\wM; \cy 2)^\pi\ar[r]&H^3(\pi;\cy 2)
}$$
Since $H_2(\wM;\cy 2) \to H_2(\widetilde P;\cy 2)$ is surjective, both vertical maps $\tilde c^*$ and $c^*$ are injective, showing uniqueness of $w$. For existence, recall from Remark~\ref{wtwo_lift} that the class $w_2(\wM)$ vanishes on the radical $R(s_M)$ of $s_M$. Hence it lies in the image of $\tilde c^*$ and a diagram chase shows that this forces $w_2(M)$ to be in the image of $c^*$.
\end{proof}
\begin{definition}\label{def:wPpair}
 The pair $(P,w)$ is called the  \emph{reduced $w_2$-type} of $M$, where $P = P(M)$ is defined above, and $w\in H^2(P;\cy 2)$ is the class constructed in
Lemma~\textup{\ref{lem:w}}. 
\end{definition}

\begin{lemma}\label{lem:reduced}
Let $(P,w)$ be the reduced $w_2$-type of $M$. Then $M$ admits a reduced normal $2$-smoothing into the homotopy pullback along $w$:
$$\xymatrix@!C
{BTOPSPIN\ar[r]^i\ar@{=}[d]&
B(P,w)\ar[r]^j\ar[d]^{\xi} &
P\ar[d]^{w}\\ 
BTOPSPIN\ar[r]^i & BSTOP\ar[r]^{w_2} & \ K(\cy 2,2)
}$$
where $w_2\colon BSTOP \to K(\cy 2, 2)$ denotes the universal second Stiefel-Whitney class. In particular, the map $\xi\colon B(P,w)\to BSTOP$ is a 3-coconnected fibration.
\end{lemma}
\begin{proof}
Since we use the homotopy pullback, the map $\xi$ is a fibration as required. Its fibre is homotopy equivalent to the fibre of $w$ which clearly has no homotopy groups in degrees~$\geq 3$. Starting with any reduced 3-equivalence $c:M\to P$, we see that a lift of $\nu_M$ to $B(P,w)$ is the same as a homotopy between the two maps $w_2\circ \nu_M$ and $w\circ c$ from $M$ to $K(\cy 2,2)$. By Lemma~\ref{lem:w} such a homotopy exists by our choice of $w$. Note that different choices of this homotopy correspond to different lifts of $\nu_M$. However, since $BTOPSPIN$, the fibre of the map $w_2$, is 3-connected, it is clear that any lift of $\nu_M$ constructed in this way is a reduced normal 2-smoothing.
\end{proof}

  The fibration $B(P,w) \to BSTOP$ constructed in Lemma \ref{lem:reduced} is the \emph{reduced normal $2$-type} of $M$,  as in Definition \ref{def:reduced normal}.  The corresponding normal bordism groups \cite{stong1} will be denoted by $\Omega_4(B(P,w))$. They can be calculated via the James spectral sequence  \cite{teichner1}, a variant of the Atiyah-Hirzebruch spectral sequence, with $E_2$-term
\[
E^2_{p,q} = H_p(P; \Omega_q^{TopSpin}(\ast)),
\]
where the notation $\Omega_*^{TopSpin}(\ast)$ means the \emph{topological} spin bordism groups.
The next result gives sufficient conditions for two 4-manifolds $M$ and $N$  to have the same reduced normal 2-type $B$. This is the first step when trying to apply Theorem~\ref{thm:reduced} to find a $s$-cobordism between $M$ and $N$. The step will be to see whether they represent the same element in $\Omega_4(B)$ for some choices of reduced normal 2-smoothings.
\begin{proposition}\label{prop:reduced}
Let  $M$, $N$ be closed oriented 4-manifolds with geometrically $2$-dimen\-sional fundamental groups.   Suppose there is an isomorphism $\alpha\colon \pi_1(M) \cong \pi_1(N)$ of their $w_2$-types and an  isometry $\beta\colon \pi_2(M)^\dag \cong \pi_2(N)^\dag$, compatible with $\alpha$. Then there is a $3$-coconnected fibration $B\to BSTOP$ admitting reduced normal $2$-smoothings $M\to B$ and $N\to B$ that induce $(\alpha,\beta)$ in the sense of Remark~\textup{\ref{rem:iso}}.
\end{proposition}
\begin{proof} 
As explained above, we may assume that $B=B(P,w)$ is constructed starting with a reduced $3$-equivalence $c_M:M\to P$. The maps $\alpha,\beta$ give a reduced $3$-equivalence $c_N \colon N\to P$. 

If $w$ lies in the image of $H^2(\pi;\cy 2)$ then $c_N^*(w) = w_2(N)$ by our assumption that $\alpha$ is an isomorphism of $w_2$-types. It follows just like in Lemma~\ref{lem:reduced} that $N$ admits a reduced normal 2-smoothing to $B$. By construction, the induced maps on $\pi_1$ and $(\pi_2)^\dag$ compose to $(\alpha,\beta)$. 

The harder case is when $w$ does {\em not} lie in the image of $H^2(\pi;\cy 2)$. This is exactly the case where $M$ and $N$ have type~(I). Since $c_M$ and $c_N$ induce the isometry 
\[
\beta=(c_N)^{-1}_*\circ (c_M)_*
\]
on reduced intersection forms, we know that $c_N^*(\tilde w)=w_2(\widetilde N)$, where $\tilde w$ is the pullback of $w$ to the universal covering $\widetilde P$. It follows that $c_N^*(w) = w_2(N) + u^*(x)$ for some $x\in H^2(\pi;\cy 2)$ and $u:P\to K(\pi,1)$ a 2-equivalance. After post-composing $c_M$ with the self-homotopy equivalence provided by the following lemma, our proof finishes just as in the  case above.
\end{proof}

\begin{lemma}\label{lem:aut}
Given $x\in H^2(\pi;\cy 2)$ and $w\in H^2(P;\cy 2)$ not in the image of $H^2(\pi;\cy 2)$, there is a self-homotopy equivalence $h$ of $P$ over $u\colon P\to K(\pi,1)$ such that $h^*(w) = w + u^*(x)$. Moreover, $h$ induces the identity map on $\pi_2(P)$. 
\end{lemma}
\begin{proof}
The homotopy classes of self-equivalences of $P$ inducing the identity on $\pi_1$ and $\pi_2$ are in bijection with the group $H^2(\pi; \pi_2(P))$. This correspondence is explicitly described in \cite[Theorem 5.2.4, p.~300]{baues1}, and explained in the proof of Lemma 11 in  \cite{hillman5}. After choosing a section of the fibration $P \to K(\pi,1)$, there is a characteristic element $\iota_P \in H^2(P; \pi_2(P))$ which maps to the identity in $\Hom_\La(\pi_2(P), \pi_2(P))$ under the natural evaluation map. For any class $\phi\in H^2(\pi; \pi_2(P))$, there exists a self-equivalence
$h_\phi\colon P \to P$ such that $$h^*(\iota_P) = \iota_P + u^*(\phi). $$
 If $ w \in H^2(P;\cy 2)$ is  not in the image of $H^2(\pi;\cy 2)$, it induces a non-zero coefficient homomorphism $w\colon \pi_2(P) \to \cy 2$, and we obtain a map $$w_*\colon H^2(\pi; \pi_2(P)) \to H^2(P;\cy 2).$$ Since $\gdim \pi \leq 2$, any class $x\in H^2(\pi;\cy 2)$ lifts to a class $\phi \in 
H^2(\pi; \pi_2(P))$ such that $w_*(\phi) = x$, by the long exact coefficient sequence. Now the required formula follows by applying the change of coefficients map $w_*$ to both sides of the relation above.
\end{proof}

\begin{proof}[The proof of Theorem C]
We can now reduce  the proof of Theorem~C to the existence of a suitable $B$-bordism, which will be studied in Section~\ref{sec:Bbordism}.  Let $M$ and $N$ be two $4$-manifolds as in that theorem. By Proposition~\ref{prop:reduced} we know that they admit reduced normal 2-smoothings in a common reduced normal 2-type $B\to BSTOP$, inducing the given maps $(\alpha,\beta)$. We know that the fundamental group in question satisfies properties (W-AA) and hence we can apply Theorem~\ref{thm:reduced}, once we show that the reduced normal 2-smoothings are $B$-bordant.

The injectivity of the assembly map $A_4$ is needed to show that such a $B$-bordism indeed exists if $M$ and $N$ satisfy the assumptions of Theorem~C. This fact is proven in Section~\ref{sec:Bbordism} and summarized in Corollary~\ref{cor:intersectionform}.
\end{proof}

\section{Thickenings and minimal models}\label{sec:minimal}

For the proof of Theorem B it is important to have a good model for each bordism class in the normal $1$-type of $M$. The central tool for this is the construction of certain minimal $4$-manifolds.

For any closed $4$-manifold $M$ with fundamental group $\pi$, we have an exact sequence
\eqncount
\begin{equation}\label{univcoeff}
0 \to H^2(\pi;\La) \to H^2(M;\La) \to \Hom_{\La}(H_2(M;\La), \La) \to H^3(\pi;\La)\to 0
\end{equation}
arising from the universal coefficients spectral sequence. Here and in the following we denote by $\La$ the group-ring $\bbZ\pi$ of the fundamental group. Using Poincar\'e duality and the Hurewicz isomorphism, we get
\[
H^2(M;\La) \cong H_2(M;\La) \cong \pi_2(M)
\]
and we can identify the middle map in the above sequence with the adjoint of the equivariant intersection form $s_M$. Therefore, we obtain the following corollary.
\begin{corollary} \label{cor:rad}
The radical $R(s_M)$ of  the intersection form $s_M$ is isomorphic  to the $\pi$-module $R(\pi):= H^2(\pi; \La)$. Similarly, the coradical of $s_M$ is always isomorphic to $H^3(\pi;\La)$. 
\end{corollary}

To construct smooth $4$-manifolds with a given fundamental group $\pi$ we start with a finite presentation $\cP$ of $\pi$ with $n$ generators and $m$ relations. Then there is a 2-complex $X(\cP)$ with a single 0-cell,  $n$ 1-cells and  $m$ 2-cells, attached according to the relations. We can turn this 2-complex to a $4$-manifold as follows. Take a single 0-handle and add $n$ oriented 1-handles leading to a boundary connected sum of $n$ copies of $S^1 \times D^3$. 

The attaching maps for $m$ 2-handles are homotopically determined by $\cP$ but they can also knot, link and have interesting framings. 
This leads to many different $4$-manifolds with fundamental group $\pi$ which we refer to as 4-dimensional {\em thickenings} of $X(\cP)$. This is a $4$-manifold that contains $X(\cP)$ as a deformation retract. This zoo of possible thickenings is drastically reduced when we double such a thickening along its boundary to produce a {\em closed}, smooth, orientable $4$-manifold. It is naturally the boundary of a 5-dimensional thickening of $X(\cP)$, namely the product of a $4$-dimensional thickening cross the unit interval.
\begin{lemma}\label{lem:thickenings} For any finite presentation $\cP$ of a group $\pi$, and class $w\in H^2(X(\cP);\cy 2)$, there is unique orientable 5-dimensional thickening $W(\cP,w)$ with deformation retraction $r$ onto $X(\cP)$ satisfying $r^*(w) = w_2(W(\cP,w))$. 
\end{lemma}
\begin{proof} 
This Lemma goes back to Wall's paper on thickenings \cite{wall-thickening} but we give a direct handle argument for the convenience of the reader: Since the handles are 5-dimensional the attaching circles of the 2-handles cannot knot or link and hence their isotopy class is determined by the presentation $\cP$. Moreover, the framing on each 2-handle is just defined mod 2. More precisely, a cocycle representative for the class $w\in H^2(X(\cP);\cy 2)$ gives a function on the $2$-cells to $\cy 2$, which can be used to vary the given framings by an element of $\pi_1(SO(3)) = \cy 2$ for each 2-handle. If one started with {\em trivial} framings (that extend over some Seifert surface) then it is not hard to see that $w$ turns into the second Stiefel-Whitney class of the thickening $W(\cP,w)$. 

Since any 2-cochain is a 2-cocyle of $X(\cP)$ it only remains to show that a 2-coboundary doesn't change the diffeomorphism class of $W(\cP,w)$. It suffices to discuss 2-coboundaries that change the framings of all 2-handles that go over a given 1-handle an odd number of times. However, a 1-handle forms $S^1 \times D^4$ together with the 0-handle and the twisting diffeomorphism of $S^1 \times S^3$ coming from $\pi_1(SO(3))$ clearly extends over this 5-manifold.
\end{proof}
We denote by $M(\cP,w)$ the boundary of a 5-dimensional thickening $W(\cP,w)$. As explained above, this is the double of a $4$-dimensional thickening. Now recall that for any space $X$ with fundamental group $\pi$ there is an exact sequence
\begin{equation} \label{eq:H^2}
\xymatrix{0 \ar[r] & H^2(\pi;\cy 2)\ar[r] & H^2(X;\cy 2) \ar[r] & H^2(\widetilde X; \cy 2)^\pi
}
\end{equation}
In particular, if $X$ is aspherical then the first map is an isomorphism which will be the case for $X=X(\cP)$ in the following discussion. In that case, the double $M(\cP,w)$ is determined by a class $w\in H^2(\pi;\cy 2)$ which is its $w_2$-type as explained in the introduction. 
 It will be implicit in the following that if a $4$-manifold has a  $w_2$-type $w\in H^2(\pi;\cy 2)$ then it does not have type~(I).

\begin{definition} \label{def:minimal}
A closed oriented $4$-manifold $M$ will be called \emph{minimal} if the intersection form on $\pi_2(M)$ vanishes, or equivalently, if $\pi_2(M) = R(s_M) \cong R(\pi)$ via the map in the above sequence~\ref{univcoeff}.
\end{definition}
The following class of groups turns out to allow minimal $4$-manifolds and several other steps in the classification of Bausmlag-Solitar groups generalize easily to this larger class.
\begin{definition}
A finitely presentable group $\pi$ is \emph{geometrically $2$-dimensional} if there exists a finite presentation $\cP$ such that the corresponding 2-complex $X(\cP)$ is aspherical. We will use the notation $\gdim \pi \leq 2$.
\end{definition}
\begin{remark} A finite presentation $\cP$ for which $X(\cP)$ is aspherical will be called a \emph{minimal} presentation. 
This is equivalent to the matrix of Fox derivatives corresponding to $\cP$ being non-singular over $\Zpi$.  
Examples of geometrically $2$-dimensional groups include free groups,  $1$-relator groups and small cancellation groups \cite{lyndon-schupp1}, provided they are torsion-free, as well as many word-hyperbolic groups (see also \cite[2.3]{hillman1}, \cite[\S 10]{hillman4}).
\end{remark}
We shall prove that  minimal $4$-manifolds with fundamental group $\pi$ exist, assuming that $\gdim \pi \leq 2$. 
\begin{lemma}\label{lem:minimal} If $\gdim \pi\leq 2$, the doubles $M(\cP,w)$ are minimal for any minimal presentation $\cP$ of $\pi$. 
\end{lemma}
\begin{proof} 
We have $M := M(\cP,w) = N\cup N$, where $N$ denotes one of the thickenings of $\cP$ described above. The long exact sequence
$$ \dots \to H_2(N;\La) \to H_2(M; \La) \to H_2(M, N; \La) \to H_1(N;\La)$$
reduces to an isomorphism in the middle because $H_i(N;\La)=0$ for $i=1,2$ since $N$ has the homotopy type of the aspherical 2-complex $X(\cP)$. This, together with excision and Poincar\'e duality, leads to isomorphisms
\[
\pi_2(M) \cong H_2(M;\La) \cong H_2(M, N; \La) \cong H_2(N,\partial N;\La) \cong H^2(N;\La) \cong H^2(\pi; \La) = R(\pi)\]
showing that $M$ is minimal.
\end{proof}
\begin{remark}\label{wtwo_lift} There is no minimal $4$-manifold $M$ with type~(I). This is because $w_2(\wM) \neq 0$ would imply that there exists a class $x\in \pi_2(M)$ with ordinary self-intersection $x\cdot x \not\equiv 0\ (\Mod 2)$. But the ordinary self-intersection number is just the coefficient of $s_M(x,x)$ at the identity element, and for a minimal $4$-manifold the form $s_M$ is zero. 
\end{remark}

\section{Stable Classification}\label{sec:stable}
Recall that two $4$-manifolds are {\em stably homeomorphic} if they become homeomorphic after connected sum with copies of $S^2 \times S^2$. It is clear that this operation preserves the fundamental group and the  $w_2$-type. Fixing the fundamental group $\pi$, the stable classification is always given by the bordism group of the {\em normal 1-type} of the $4$-manifolds, see \cite[p.~711]{kreck3}. As for the normal 2-type explained in Section~\ref{sec:strategy}, this is a 2-coconnected fibration $B\to BSTOP$ that admits a lift of the normal Gauss map that is a 2-equivalence. 

The easiest case is type~(I) where we have the following  application of the methods of  \cite{kreck3}. For any closed, oriented $4$-manifold $M$, let  $c\colon M \to K(\pi,1)$ denote a classifying map of its universal covering, and let $c_*[M]\in H_4(\pi,\bbZ)$ denote the image of its fundamental class.
\begin{lemma}
Two closed oriented $4$-manifolds $M_1$ and $M_2$ of  type~\textup{(I)} are stably homeomorphic if and only if they have the same fundamental group, signature and Kirby-Siebenmann invariant, and $c_*[M_1] = c_*[M_2] \in H_4(\pi; \bbZ)$.
\end{lemma}
\begin{proof}
For  type~(I), the normal 1-type is just $BSTOP \times \Bpi$ and we can use the Atiyah-Hirzebruch spectral sequence  (and the well-known values 
$\Omega_4^{STOP}(\ast)= \bbZ \oplus\cy 2$ and $\Omega_q^{STOP}(\ast) = 0$ for $1\leq q \leq 3$) to compute
\[
\Omega^{STOP}_4(\Bpi) \cong \bbZ \oplus \cy 2\oplus H_4(\pi;\bbZ)
\]
with signature, Kirby-Siebenmann invariant and fundamental class $c_*[M]$ giving the isomorphism.
\end{proof}

\begin{remark}\label{signature_remark}
Recall that the signature of a closed, oriented $4$-manifold $M$ refers to the signature of the intersection form on $H^2(M; \bbZ)$, given by the cup product pairing and evaluation on $[M]\in H_4(M;\bbZ)$. 
If a 2-equivalence $c\colon M \to K(\pi,1)$ has the property that $0=c_*[M]\in H_4(\pi;\bbZ)$, then $\sign(M)$ is equal to the signature of the form $s_M\otimes_{\Zpi}\bbZ$. This is because the image $H^2(\pi;\bbZ) \overset{c^*}{\to} H^2(M;\bbZ)$ is totally isotropic under the cup product pairing by the following observation for $x,y\in H^2(\pi;\bbZ) $:
 \[
 \langle c^*(x) \cup c^*(y), [M] \rangle = \langle x\cup y , c_*[M] \rangle =0
 \]
 This remark applies in particular to all $4$-manifolds with geometrically 2-dimensional fundamental group.
 \end{remark}

\begin{corollary}
Let $M$ be a closed, oriented  $4$-manifold  with geometrically 2-dimensional fundamental group and  type~\textup{(I)}. Then $M$ is stably homeomorphic to $M_0 \Sharp N$, where $M_0$ is minimal and $N$ is a closed simply-connected $4$-manifold.
\end{corollary} 
\begin{proof} Let $M_0$ denote a minimal smooth $4$-manifold constructed in Lemma~\ref{lem:minimal}. Adding copies of $\CP^2$, or the Chern manifold, with appropriate orientations we can arrange that  $M$ and $M_0 \Sharp N$ have the same signature and Kirby-Siebenmann invariant, and both have type~(I). By the above lemma and remark, they are stably homeomorphic.
\end{proof}

The following proves the first part of Theorem B:

\begin{corollary} \label{cor:free}
If $M$ is a closed, oriented  $4$-manifold  with geometrically 2-dimensional fundamental group $\pi$, then $\pi_2(M)^\dag = \pi_2(M)/R(s_M)$ is a finitely-generated stably-free $\Zpi$-module, where $R(s_M)$ is the radical of the intersection form $s_M$ on $\pi_2(M)$.
\end{corollary}
\begin{proof} If $M$ is any closed $4$-manifold with $\pi_1(M) = \pi$, we can form $M \Sharp \CP^2$ to obtain type~(I). Therefore, $M  \Sharp \CP^2$ is stably homeomorphic to a manifold of the form $M_0 \Sharp N$, where $N$ is simply-connected and $M_0$ is minimal, i.e. $\pi_2(M_0) \cong R(s_{M_0})$. The result now follows from the exact sequence~(\ref{univcoeff}), which is split short exact in this case.
\end{proof}

\begin{lemma}\label{stable_bordism} For $w_2$-type $w\in H^2(\pi;\cy 2)$, with $\gdim\pi\leq 2$, the bordism groups of the normal $1$-type are given by
$$\Omega_4(B(\pi, w)) \cong  8\bbZ  \oplus H_2(\pi;\cy 2).
$$
The first invariant is the signature (always divisble by 8) and modulo the action of the automorphism group of $\pi$ on the second factor, this gives the stable homeomorphism classification of closed oriented manifolds with fundamental group $\pi$ and $w_2$-type $w$.
\end{lemma}
\begin{proof} 
Just as in Lemma~\ref{lem:reduced}, one proves that the normal 1-type is given by the homotopy pullback
$$\xymatrix@!C
{BTOPSPIN\ar[r]\ar@{=}[d]&
B(\pi,w)\ar[r]\ar[d]^{\xi} &
\Bpi \ar[d]^{w}\\ 
BTOPSPIN\ar[r] & BSTOP\ar[r]^{w_2} & \ K(\cy 2,2)
}$$
where $w_2\colon BSTOP \to K(\cy 2, 2)$ denotes the universal second Stiefel-Whitney class. The corresponding normal bordism groups will be denoted by $\Omega_4(\pi,w)$. They can be calculated via the James spectral sequence \cite{teichner1}, with $E_2$-term
\[
E^2_{p,q} = H_p(\pi; \Omega_q^{TopSpin}(\ast)).
\]
Recall that $\Omega_q^{TopSpin}(\ast) = \bbZ, \cy 2, \cy 2, 0,  \bbZ$, in the range $0 \leq q\leq 4$. Since $\pi$ is 2-dimensional, all groups with $p>2$ are zero. In particular, 
there are no differentials affecting the line $p+q=4$. The bordism class represented by the $E8$-manifold accounts for the term $E^2_{0,4} = 8\bbZ$. The term $E^2_{2,2} = H_2(\pi;\cy 2)$ and there are no other terms on the line $p+q=4$. It follows from \cite{teichner1} that the signature of any $4$-manifold with this normal 1-type is divisible by $8$ and hence the result follows.
\end{proof}

\section{Detecting $B$-bordism classes}\label{sec:Bbordism}

In this section we fix a geometrically 2-dimensional group $\pi$ that is going to be the fundamental group of all our $4$-manifolds below. We want to compute the bordism group $\Omega_4(B(P,w))$ using the James spectral sequence. For this we 
compute the homology of $P$ from the Leray-Serre spectral sequence of the fibration
\[
K(A, 2) \to P \xrightarrow{u} \Bpi,
\]
where  $A:=\pi_2(P)$ is a finitely-generated, stably-free $\La$-module by Corollary~\ref{cor:free}.  Hence $A$ is a countably-generated free abelian group. Let $\Gamma(A) = H_4(\widetilde P;\bbZ)$ denote Whitehead's $\Gamma$-functor (see \cite[\S~5]{whitehead1}), and for any $\La$-module $L$ we will use the notation $ L_\pi:= L\otimes_\La \bbZ $ for the cofixed set of $L$. 
\begin{lemma}\label{lem:homology}
For $i = 0, 1, 2, 3, 4$, the homology groups $H_i(P; \bbZ)$ are given by
$$\bbZ ,\ H_1(\pi;  \bbZ), \ H_2(\pi; \bbZ) \oplus A_\pi, \ 0, \ \Gamma(A)_\pi$$
\end{lemma}
\begin{proof}
Since $A$ is a stably-free $\La$-module, so is $\Gamma(A)$ (compare
\cite[Lemma 2.2]{hk2}). This means that $H_i(\pi; A) = 0 $ and $H_i(\pi; \Gamma(A)) =0$ for $i > 0$. Further details will be left to the reader. 
\end{proof}

In order to compute the  differentials in the James spectral sequence below, we will need the following result. Let $u \colon P \to K(\pi,1)$ denote the classifying map of the universal covering.
\begin{lemma} \label{lem:Sq}
For any class $w\in H^2(P;\cy 2)$, 
the map $Sq^2_w\colon H^2(P;\cy 2) \to H^4(P;\cy 2)$, given by $Sq^2_w(x) = x\cup x + x \cup w$, has kernel generated by $\{w, u^*H^2(\pi;\cy 2)\}$.
\end{lemma}
\begin{proof} Let $p\colon \widetilde P \to P$ denote the projection of the universal covering and $\tilde w = p^*(w)$. 
If $\tilde w \neq 0$, then $u^*(y) \cup w = 0$, for all $y \in H^2(\pi;\cy 2)$, since $H^2(\pi; H^2(\widetilde P;\cy 2)) =0$ in the $(2,2)$ position of the spectral sequence for the  universal covering $\widetilde P \to P$. 
We therefore have the commutative diagram 
$$\xymatrix{0 \ar[r]&H^2(\pi;\cy 2) \ar[r]^{u^*}\ar[d]& H^2(P;\cy 2)\ar[r] \ar[d]^{Sq^2_w}& H^2(\widetilde P;\cy 2)\ar[d]^{Sq^2_{\tilde w}}\cr
&0 \ar[r]& H^4(P;\cy 2)\ar[r] & H^4(\widetilde P;\cy 2)
}$$
since $H^4(\pi;\cy 2) =0$. 

The map induced by $Sq^2_{\tilde w}$ on $H^2(\widetilde P;\cy 2)$ has kernal $\la \tilde w\ra$ since the space $\widetilde P = K(A,2)$  is the Eilenberg-Maclane  space of a countable direct sum of copies of $\bbZ$. Its homotopy type is the colimit
$$K(A,2) = K\big (\bigoplus^\infty \bbZ, 2\big ) \simeq \colim_{k \to \infty}\bigr (\prod_{i=1}^{k} \CP^{\infty}\bigr )$$
 of products of finitely many copies of $\CP^{\infty}$.
\end{proof}

\begin{lemma}\label{bordismtwotwo} Let $(P, w)$ be a  reduced  $w_2$-type with $\gdim \pi_1(P) \leq 2$. Then there is an injection  
$$ \Omega_4(B(P,w))\subseteq  \bbZ \oplus \cy 2 \oplus H_2(\pi;\cy 2)\oplus H_4(P;\bbZ)$$
 detecting the bordism groups of the  reduced  normal $2$-type $B(P,w)$.
The invariants are the signature, the $KS$-invariant, an invariant in $H_2(\pi;\cy 2)$, and the fundamental class $c_*[M] \in H_4(P;\bbZ)$.
\end{lemma}
Later we will define and investigate the bordism invariant in $H_2(\pi;\cy 2)$.
\begin{proof}
The argument splits naturally into two cases, depending on whether $(B(P,w))$ is of type~(I) or not. In the latter case, the  reduced  normal $2$-type pulls back from the normal 1-type. Since the $k$-invariant for $P$ lies  in $H^3(\pi;\pi_2(P))=0$, the fibration $P \to K(\pi, 1)$ has a section which gives a direct sum splitting  of the bordism groups. The result then follows almost directly from the homology computation in Lemma~\ref{lem:homology}. The James spectral sequence has  $E_2$-term
\[
E^2_{p,q} = H_p(P; \Omega_q^{TopSpin}(\ast)),
\]
and the only subtlety is the $d_2$-differentials that start in the $E_{4,i}$ spots, where $i=0,1$. For $i=1$, it is given by the dual of the map in Lemma~\ref{lem:Sq} above. This Lemma implies that the cokernel of $d_2$ is $H_2(\pi;\cy 2)$. For $i=0$, one needs to compose in addition with the reduction map from $H_4$ with integral to $\cy 2$ coefficients. The resulting $d_2$-differential has a kernel inside $H_4(P;\bbZ)$ which is exactly the image of the inclusion described in the statement of our lemma.

Now assume that $(P,w)$ is of type~(I). By  Lemma \ref{lem:Sq} the $E_3$-term of the James spectral sequence has at the $(2,2)$ position $H_2(\pi;\mathbb Z/2) \oplus \mathbb Z/2\, [\alpha]$, where $\alpha \in H_2(P;\mathbb Z/2)$ is a spherical class such that $ \langle w,\alpha \rangle  \ne 0$.  We claim that the component in $\mathbb Z/2 [ \alpha]$ is determined by the Kirby-Siebenmann invariant. It is enough to find a bordism class with signature zero and trivial image of the fundamental class which represents $\alpha$ and has non-trivial Kirby-Siebenmann invariant. 
 For this we consider $N:= \CH \Sharp  (- \CP^2)$, where $\CH= \ast\CP^2$ is the Chern manifold with non-trivial Kirby-Siebenmann invariant \cite{freedman-quinn1}. The manifold $N$ is homotopy equivalent to $ \CP^2 \Sharp  (-\CP^2)$ which is an $S^2$-bundle over $S^2$. The bundle projection translates to a map $g\colon N \to S^2$ which sends the classes of square $\pm 1$ in $H_2(N)$ to generators in $H_2(S^2)$. 

  Since $\alpha$ is a spherical class, there is a map $h\colon  S^2 \to P$ representing $\alpha$. The bordism class we are looking for is represented by $(N, hg)$. By construction the pullback of $w$ under this map is $w_2(N)$ and so it gives an element in our bordism group, which has the desired properties.
 \end{proof}

\begin{remark}\label{rem:comparison}
If $(P,w)$ is not of type~(I) then the above proof shows that the bordism groups of the  reduced  normal $2$-type $B(P,w)$ are detected by the natural map to the bordism groups of the normal 1-type $B(\pi,w)$ and the fundamental class $c_*[M] \in H_4(P;\bbZ)$. In particular, it follows from Lemma~\ref{stable_bordism} that the signature is divisible by~8 and the Kirby-Siebenmann invariant is determined by the other invariants. For type~(I) we proved that the signature can be any integer and the Kirby-Siebenmann invariant is independent from all other invariants. 
\end{remark}

Next we will show that the bordism invariant in $H_2(\pi;\cy 2)$  given in Lemma \ref{bordismtwotwo} is detected by the other invariants. This was used in the proof of Theorem~C. 

 Recall that a closed oriented $4$-manifold $M$ determines
 a reduced $3$-equivalence
$c\colon M \to P$ (see Definition~\ref{def:reduced}), where $(P,w)$ is the reduced  $w_2$-type $(P, w)$ of $M$ (see Lemma \ref{lem:w} and Definition \ref{def:wPpair}). Let $B(M)=B(P,w)$ denote the resulting reduced normal $2$-type of $M$.  
\begin{definition}
 We define the subset of \emph{normal structures} in $\Omega_4(B(M))$, denoted by $
 \Omega_4(B(M))_M$. It consists  of the normal bordism classes $(N,f)$ over $B(M)$,  with $f_*[N] = c_*[M]$, $\sign(N) = \sign(M)$,  and $KS(M) = KS(N)$. In other words it is the subset of the reduced bordism group, which is the fibre over $(c_*[M], \sign(M), KS(M))$ of the map to $H_4(P)\oplus \mathbb Z \oplus \mathbb Z/2$ given by the image of the fundamental class, the signature and the Kirby-Siebenmann invariant. 
 \qed
\end{definition}
We stress that $ \Omega_4(B(M))_M$ is a subset and not a subgroup. 
This subset is non-empty, since it contains $[M, \hat c]$.
Now we define a map
$$\theta\colon \Omega_4(B(M))_M \to L_4(\Zpi)$$
as follows: for any element $[N,f]\in \Omega_4(B(M))_M$, we do surgeries until  $f$ is $2$-connected and let 
$$V:= \ker(\pi_2(N) \to \pi_2(B(M)))$$
Since $f^*(w)=w_2(N)$, it follows that $w_2(\widetilde N)$ is zero when restricted to $V$. Moreover, $R(s_N) \subseteq V$ since the intersection form is non-singular on $\pi_2(B)$ by construction. 
\begin{lemma} The restriction of $s_N$ to $V$ induces a non-singular, even form $\lambda_{N,f}$ on $V^\dag = V/R(s_N)$, which is a finitely-generated stably free $\La$-module.
\end{lemma}
\begin{proof} 
Since $\pi_2(B) \cong \pi_2(P) \cong \pi_2(M)^\dag$ is a stably-free $\La$-module, we can split $V = R(\pi)\oplus A$, where $A$ is also a
stably-free $\La$-module. Since $ j_*f_*[N] = c_*[M]$ it follows that the restriction of $s_N$ to $A \cong V^\dag$ is non-singular (this is the usual argument for surgery kernels \cite[Lemma 2.2]{wall-book}).
\end{proof}

\begin{remark}
 If $[N,f] \in \Omega_4(B(M))_M$ and $f$ is a reduced $3$-equivalence, then $\sMdagg{N} \cong  \sMdag$. The reason is that the image of the fundamental class of a reduced $3$-equivalence determines the reduced intersection form.
\end{remark}

\begin{definition}\label{def:theta}
We define the map
$\theta\colon \Omega_4(B(M))_M \to L_4(\Zpi)$ by setting $\theta(N,f) = [\lambda_{N,f}] \in L_4(\Zpi)$.
\end{definition}
\begin{lemma}
The map $\theta$ is well-defined.
\end{lemma}
\begin{proof} If $[N_1, f_1] = [N_2,f_2]$, then by \cite[Cor.~3]{kreck3} the manifolds $N_1$ and $N_2$ become homeomorphic over $B$, after connected sum with copies of $S^2\times S^2$. It follows that the non-singular even forms $\lambda_{V_1}$ and $\lambda_{V_2}$ become isometric by adding hyperbolic forms on free $\La$-modules.
\end{proof}

If $X$ is a closed, simply-connected manifold, then we have a stabilization map
\[
j_X \colon \Omega_4(B(M)) \to \Omega_4(B(M\Sharp X)),
\]
defined by sending an element $(N,f)$ to $(N\Sharp X, f_X)$, where the new reference map $f_X \colon N \Sharp X \to B(M\Sharp X)$ is given by the composition
$$N \Sharp X \to N \vee X \xrightarrow{f\vee id_X} B(M) \vee X \to 
B(M)\vee B(X) \to B(M\Sharp X).$$
It is clear that the stabilization map induces a map of the normal structure subsets.
By construction of the map $\theta$, we have the following:

\begin {lemma} Let $X$ be a closed simply connected manifold. Then $\theta$ commutes with the stabilization map $j_X\colon \Omega_4(B(M))_M \to \Omega_4(B(M\Sharp X))_{M\Sharp X}$.
\end {lemma}

Recall from Lemma~\ref{bordismtwotwo} that  we have an injection of sets
$$
\rho_M :\Omega_4(B(M))_M \to H_2(\pi_1;\mathbb Z/2)
$$
defined by projecting the element
$$
[N,f] \mapsto [N,f] - [M,c],
$$
in the second filtration subgroup of $\Omega_4(B(M))$ to $E^{\infty}_{2,2} \subseteq H_2(\pi;\cy 2) \oplus \cy 2$, and then projecting further into $H_2(\pi;\cy 2)$.
By construction, the map $\rho$ also commutes with the stabilization map  $\Omega_4(B(M))_M\to \Omega_4(B(M\Sharp X))_{M\Sharp X}$ as defined above.
We want to relate $\rho$ to the assembly map
$$A_4\colon H_4( \Bpi; \bbL_0(\bbZ)) \to L_4(\Zpi)$$
as described in \cite[\S 1]{hmtw1}.  The domain of this assembly map is given by
 $$H_4( \Bpi; \bbL_0(\bbZ)) = H_0(\pi;\bbZ) \oplus H_2(\pi;\cy 2)$$
and we only need the restriction
$$ \kappa_2\colon
H_2(\pi;\cy 2) \to L_{4}(\Zpi)$$
of $A_4$ to the second summand. By comparing with the trivial group, it is easy to see that $A_4$ is injective if and only if $\kappa_2$ is injective, so that is part of our assumption (W-AA) in Theorem~C.

\begin {lemma}\label{lem:theta_relation} In the above setting there is a commutative diagram
$$\xymatrix{ \Omega_4(B(M))_M\ar[dr]^{\rho_M}\ar[rr]^{\theta_M} &&\widetilde L_4(\bbZ\pi)\cr &H_2(\pi;\cy 2) \ar[ur]_{\kappa_2}&}$$
\end{lemma}
\begin{proof} The case where $M$ is a minimal spin manifold follows from Davis \cite[Thm.3.10]{jdavis2}.  By the signature theorem,  for any degree one normal map $f\colon  N \to (M,\nu_M)$ the signature of $N$ is equal to that of $M$. Thus  there is a bijection between the  set of degree one normal maps to $(M, \nu_M)$ and $H^2(M;\mathbb Z/2)$ \cite{wall-book}, \cite{kirby-taylor1}. Since the Kirby-Siebenmann invariant of a spin manifold is determined by the signature, also this invariant agrees with that of $M$. Davis starts with such a degree $1$ normal map $f\colon  N \to (M,\nu_M)$ (which corresponds to an element $\beta \in H^2(M;\cy 2)$),  and chooses spin structures on $M$ and $N$ so that he can consider the elements $[M,id]$ and $[N,f]$ in $\Omega_4^{TopSpin}(M)$. Then he considers the Atiyah-Hirzebruch spectral sequence computing $\Omega_4^{TopSpin}(M)$ and shows that there are spin structures on $M$ and $N$ such that $\gamma := [M,id] - [N,f]$ sits in the filtration subgroup $F_{2,2}$. Furthermore he shows that $\gamma $ maps to  $\beta \cap [M] \in H_2(M;\mathbb Z/2) = E^\infty_{2,2}$. This together with  Wall's characteristic class formula \cite{wall-1976}, \cite{taylor-williams1} implies that $\kappa _2 (u_*(\gamma)) = \theta (N,f)$. 

Now we compare this information with the corresponding information when we pass from $M$ to $B(M)$. If $M$ is minimal, then $B(M) = K(\pi,1) \times BTOPSPIN$. By construction of our map $\theta$ and the classical surgery obstruction we have $\theta (N,f) = \theta (N,cf)$. We conclude that for $M$ minimal
$$
\theta [N,cf]= \kappa_2(\rho[N,cf]).
$$

We summarize these considerations: If $M$ is a minimal spin manifold and $\beta \in H_2(\pi;\mathbb Z/2)$ there is an $\alpha \in \Omega_4(B(M))_M$ with $\rho(\alpha) = \beta$ and $$\theta (\alpha) =\kappa _2(\rho (\alpha)).$$ This implies that $\rho$ is surjective and so, since it is injective,  a bijection to $H_2(\pi;\mathbb Z/2)$. From this we have the required formula for minimal manifolds. 

Since the maps $\rho$ commute with stabilization by connected sum with any simply connected manifold $X$, $\rho$ is a bijection for $M_0\Sharp X$ and the stabilization map $$j_X\colon \Omega_4(B(M))_M \to \Omega_4(B(M\Sharp X))_{M\Sharp X}$$ is a bijection. Since the maps $\theta$ also commute with stabilization,  the relation $\theta (\alpha) =\kappa _2(\rho (\alpha))$ holds for $M_0\Sharp X$ as well.

Next we remark that an orientation-preserving homeomorphism $h\colon M \to M'$ over $K(\pi,1)$ induces a map
$$ \Omega_4(B(M))_M \approx  \Omega_4(B(M'))_{M'}$$
 on the subset of normal structures, 
induced by composing the reference maps with $h$, or $(N,f) \mapsto (N, h\circ f)$.
More precisely, it is clear that the fundamental class $c_*[M] \mapsto c'_*[M']$, and the conditions on signature and $KS$-invariant are preserved by orientation-preserving homeomorphisms.  

Finally, if $M_1$ is arbitrary and $KS(M) =0$, there exist integers $(r,s)$ and $(r',s')$ such that
$M=M_0 \Sharp X$ is homeomorphic to $M'=M_1\Sharp X'$, where
$X = \Sharp_r \CP^2 \Sharp_s(-\CP^2)$, and $X'=\Sharp_{r'}\CP^2 \Sharp_{s'} (-\CP^2)$. 
Again commutativity of the maps under stabilization with $1$-connected manifolds implies the Lemma. If $KS(M) =1$, we replace one of the $\CP^2$'s by the Chern manifold $\CH$.
\end{proof}

\begin{corollary}\label{cor:fundclass}
Suppose that $[N,f]$, is an element in $\Omega_4(B(M))$, with $f$  a  reduced $3$-equivalence such that $\sign (N) = \sign (M)$ and $KS(N) = KS(M)$. If $f_*[N] = c_*[M]$ and $\kappa_2\colon H_2(\pi;\cy 2) \to L_4(\Zpi)$ is injective, then $[N,f]=[M,c] \in \Omega_4(B(M))$.
\end{corollary}

\begin{proof} If $f_*[N] = c_*[M]$, the intersections forms are isometric since $f$ is a reduced $3$-equivalence, and so $\theta [N,f] = \theta [M,c]=0$. By the previous Lemma, $\kappa _2(\rho [N,f] )= 0$ and so $\rho[N,f] = 0$ since $\kappa _2$ is injective. But $\rho[N,f]$ is defined as the projection of the difference element $[N,f]-[M,c]$ into $H_2(\pi;\cy 2)$. Since by Lemma \ref{bordismtwotwo} the bordism class is determined by the signature, the Kirby-Siebenmann obstruction, the image of the fundamental class and $\rho$, we have $[N,f]=[M,c]$.
\end{proof}

The next step in our bordism calculation is to control the image of the fundamental class $c_*[M] \in H_4(P)$ by the reduced intersection form $\sMdag$. 

\begin{theorem}\label{thm:fundamentalclass}
Two  reduced  $3$-equivalences  $c_M\colon M\to P$ and $c_N\colon N\to P$ satisfy 
\[
(c_M)_*[M] = (c_N)_*[N] \in H_4(P;\bbZ)
\]
 if and only if $(c_N)_*^{-1} \circ (c_M)_*:\pi_2(M)^\dag \to \pi_2(N)^\dag$ induces an isometry of reduced intersection forms.
\end{theorem}

 The proof of this result will be given at the end of this section. Since the signature is determined by the intersection form on $\pi_2(M)$, we conclude from this result and Corollary \ref{cor:fundclass} the following result. 
\begin{corollary}\label{cor:intersectionform} Let $M$ and $N$ be closed oriented $4$-manifolds with the same Kirby-Siebenmann invariants. Suppose that $\kappa_2\colon H_2(\pi_1(M);\cy 2) \to L_4(\Zpi_1(M))$ is injective. Let $\alpha : \pi_1(M) \to \pi_1(N)$ be an isomorphism of $w_2$-types and $\beta: \pi_2(M)^\dag \to \pi_2(N)^\dag$ an $\alpha$-compatible isometry. Then there are reduced  $2$-smoothings  $M\to B(M)$ and $N\to B(M)$ compatible with  $(\alpha, \beta)$ which are bordant in $\Omega_4(B(M))$.
\end{corollary}

\begin{proof}
Proposition~\ref{prop:reduced} implies that there are reduced 2-smoothings into the same 3-coconnected fibration $B\to BSTOP$ which we may assume to equal $B(M)$. Lemma~\ref{bordismtwotwo} gives the invariants that control the bordism class in $\Omega_4(B(M))$. All these invariants are controlled by our assumptions and Theorem~\ref{thm:fundamentalclass}, Corollary~\ref{cor:fundclass}. 
\end{proof}

The remaining part of this section is devoted to the proof of Theorem \ref{thm:fundamentalclass}.
There is a commutative diagram
$$\xymatrix{H_4^{lf}(\widetilde P;\bbZ)^\pi \ar[r]^(0.3){\omega}&  \Hom_{\bbZ}(H^2_{cp}(\widetilde P;\bbZ), H_2(\widetilde P;\bbZ))^\pi\ar@{=}[d]\cr
H_4(P;\bbZ) \ar[r]^(0.3){\omega} \ar[u]^{tr}& \Hom_{\La}(H^2(P;\La), H_2(P;\La))
}$$
where $tr\colon H_4(P;\bbZ) \to H^{lf}_4(\widetilde P;\bbZ)$ denotes the transfer
map induced by the universal covering $\widetilde P \to P$. The image of
$tr(c_*[M])$ under the top horizontal slant product  map $\omega$ is just the inverse of the adjoint of the
equivariant intersection form $\sMdag$.
\begin{lemma}\label{injective}
The composition $\omega\circ tr$ is injective.
\end{lemma}
It is enough to prove this injectivity after stabilizing $M$ by connected sum with copies of $S^2 \times S^2$, so we may assume that $A = \pi_2(P)$ is a finitely-generated, free $\La$-module. Let $\{a_i\}$ denote a $\bbZ$-basis for $A$. In our applications, each $a_i = ge_j$, for some $g\in \pi$, where $\{e_j\}$ denotes a given $\La$-basis for $A$. This is the natural underlying $\bbZ$-basis for a free, based $\La$-module. 

Following \cite[p.~62]{whitehead1}, we define
$$A^* = \{ \phi \colon A \to \bbZ \vv  \phi(a_i) = 0 \text{\ for\ almost\ all\ }   i\}$$
Let $\{a_i^*\}$ denote the dual basis for $A^*$. We say that a homomorphism $f\colon A^* \to A$ is \emph{admissible} if $f(a^*_i) = 0$ for almost all $i$, and that $f$ is \emph{symmetric} if $a^*fb^* = b^*fa^*$ for all $a^*, b^* \in A^*$.
\begin{lemma}[{\cite[p.~62]{whitehead1}}]
$$\Gamma(A) \cong \{ f \colon A^* \to A\vv f \text{\ is\ symmetric\  and\ admissible}\}$$
\end{lemma}
\begin{proof}[The proof of Lemma \ref{injective}]
Suppose now that $A = \La^r$, and notice that $\Hom_\La(A, \La) \cong A^*$.  Then we have a commutative diagram
$$\xymatrix{\Gamma(A)^\pi \ar@{^{(}->}[r] &H_4^{lf}(\widetilde P;\bbZ)^\pi \ar[r]^(0.4){\omega}&  \Hom_{\bbZ}(A^*, A)^\pi\cr
\Gamma(A)_\pi \ar@{=}[r] \ar[u]_N&H_4(\widetilde P;\bbZ)_\pi \ar[r]^(0.4){\approx} \ar[u]^{N}& 
\Hom^{a}_{\bbZ}(A^*, A)_\pi\ar[u]_N}
$$
where $\Hom^{a}$ denotes the  admissible homomorphisms, and the norm maps $N\colon L_\pi \to L^\pi$ are formally defined for any $\La$-module by applying the operator $N = \sum\{g \vv g\in \pi\}$ (this makes sense only if the sum is actually finite when $N$ is applied to elements of $L$). Now the point is that the right-hand norm map in the diagram is a direct sum of the norm maps 
$$N\colon \Hom^{a}_{\bbZ}(\La^*, \La)_\pi \to  \Hom_{\bbZ}(\La^*, \La)^\pi $$
 It is convenient to identify $\La^* \cong \La$, and then express
$$ \Hom_{\bbZ}(\La^*, \La) \cong \Hom_{\bbZ}(\La, \La) \cong \bigoplus_{g\in \pi}\Hom_{\bbZ}(\La, \bbZ)$$
where the induced $\pi$-action on the right-hand side permutes the copies of $\Hom_{\bbZ}(\La, \bbZ)$ by the formula $(h\cdot\phi)_g= h\phi_{h^{-1}g}$. If we restrict to the $\pi$-fixed set of this action, then the component at $g=e$ determines all of the other components.
Therefore, projection on the component of $g=e$ gives an isomorphism
$$ (\bigoplus_{g\in \pi}\Hom_{\bbZ}(\La, \bbZ) )^\pi \cong \Hom_{\bbZ}(\La, \bbZ)$$
Similarly, if we restrict to admissible maps and project to the co-fixed set, the maps concentrated at the identity component represent the equivalence classes. Therefore
$$\Hom^a_{\bbZ}(\La^*, \La)_\pi  \cong (\bigoplus_{g\in \pi}\Hom^a_{\bbZ}(\La, \bbZ) )_\pi \cong \Hom^a_{\bbZ}(\La, \bbZ)$$
Let $\phi = (\phi_g)$ be an admissible homomorphism with $\phi_g = 0$ unless $g=e$. 
The norm map is now given by the formula
$$(N\phi)_g = \sum_h (h\cdot\phi)_g = \sum_h  h\phi_{h^{-1}g} = g\phi_e$$
After projection to the identity component, we see that the right-hand norm map in the diagram is just a direct sum of standard inclusions
$$\Hom^a_{\bbZ}(\La, \bbZ) \subset \Hom_{\bbZ}(\La, \bbZ)$$
which is certainly injective. It follows that $\omega\circ N$ is injective.

Under our assumptions on $P$, we have an isomorphism $H_4(\widetilde P;\bbZ)_\pi \cong H_4(P;\bbZ)$ by projection from the universal covering. Therefore the norm map $N\colon H_4(\widetilde P;\bbZ)_\pi \to 
H_4^{lf}(\widetilde P;\bbZ)^\pi$ may be identified with the transfer
$tr$. Therefore $\omega\circ tr$ is injective.
\end{proof}
\begin{remark} Another way to express this conclusion about $N$ is to identify
$$\La\cong  \Hom^{a}_{\bbZ}(\La^*, \La)_\pi \to  \Hom_{\bbZ}(\La^*, \La)^\pi \cong \widehat \La$$
where $\widehat \La$ denotes the ring of infinite $\bbZ$-linear combinations of elements of $\pi$. One can verify 
that the norm map 
corresponds to the natural inclusion $\La \subset \widehat\La$, which is just $\bigoplus_{g\in \pi}\bbZ \subset \prod_{g\in \pi}\bbZ$.
\end{remark}
\begin{proof}[The proof of Theorem \ref{thm:fundamentalclass}] 
If $(c_1)_*[M_1] = (c_2)_*[M_2]$, we get $\sMdagg{M_1} \cong \sMdagg{M_2}$ by applying the map $\omega\circ tr$. The converse holds since $\omega\circ tr$ is injective  by Lemma \ref{injective}.
\end{proof}

\section{Baumslag-Solitar groups}\label{sec:BS}
Our goal in this section is to  establish the properties (W-AA) from Definition~\ref{def:WAA} for these groups.
The Baumslag-Solitar groups $\pi=B(k)$ are 1-relator groups, so 
 the presentation $2$-complex $X$  is a $K(\pi,1)$,  \cite{lyndon-schupp1}. It follows that $H^2(X; \Zpi) = H^2(\pi; \Zpi)$. 
 \begin{remark}
For the ``exceptional" cases $k = 0, 1, -1$ we just have $X=S^1$, the torus $T^2$, or  the Klein  bottle.
\end{remark}
 The chain complex $C(\wX)$ for $k \neq 0$ has the form (compare \cite[Lemma 4.3]{friedl-teichner1})
 $$0 \to \La \xrightarrow{\bd_2} \La \oplus \La \xrightarrow{\bd_1} \La \xrightarrow{\epsilon} \bbZ \to 0$$
 where $\La = \Zpi$ denotes the integral group ring, and $\epsilon$ is the augmentation map. The boundary map $\bd_1 =({1-a},{1-b})$ and the boundary map $\bd_2 = \binom{\bd_a}{\bd_b}$
 is given by the Fox derivatives  \cite{fox_1953} of the relation $aba^{-1}b^{-k}$
 $$\bd_a = 1 - aba^{-1}, \quad \bd_b = a - aba^{-1}b^{-k}\left ( {\frac{b^k - 1}{b-1}}\right )$$

From this complex, one can compute the homology  of $B(k)$. Note the convention $(\bd_1\circ\bd_2)(v) = v\cdot \bd_2 \cdot \bd_1$ expressing the composition in terms of right matrix multiplication. 
\begin{lemma} \label{lem:homology2}
For $\pi = B(k)$, we have the following (co)homology groups.
\begin{enumerate}
\item $H_1(\pi;\bbZ) = \bbZ \oplus \cy{(k-1)}$, $H_i(\pi;\bbZ) = 0$ for $i > 2$.
\item $H_2(\pi;\bbZ) = \bbZ$ if $k=1$, and
$H_2(\pi;\bbZ) = 0$ otherwise.
\item $H^2(\pi;\cy 2) = \cy 2$, if $k$ is odd, and $H^2(\pi;\cy 2) = 0$, if $k$ is even.
\item $H^i(\pi;\Zpi) = 0$ for $i \neq 2$ and $k \neq 0$.
\item $R(\pi) = H^2(\pi; \Zpi)$ is free abelian, and surjects onto $\bbZ[1/k]$ if $k \neq 0$.
\end{enumerate}
\end{lemma}
\begin{proof} By \cite[Cor.~2]{mihalik-tschantz1} the group $H^2(\pi; \Zpi)$ is free abelian for 1-relator groups, and as observed in \cite[Lemma 4.3]{friedl-teichner1} it surjects onto $\bbZ[1/k]$. The other parts will be left to the reader.
\end{proof}
\begin{remark} It is an interesting question whether $\pi_2(M)$ is always a free abelian group, for any closed $4$-manifold $M$. This seems to be the same as asking whether $H^2(\pi;\Zpi)$ is always free abelian for any finitely-presented group $\pi$. The latter is a well-known question in group theory.  
\end{remark}

We will need to compute the Whitehead group $\wh(\pi)$ and the surgery obstruction groups $L_4(\Zpi)$, $L_5(\Zpi)$, for our fundamental groups $\pi=B(k)$. We will use the well-known fact that the $K$-theory and $L$-theory functors commute with direct limits of rings (with involution).
The idea is that these functors are defined in terms of $n\times n$ matrices over rings, and such matrices live at a
finite stage of the direct limit. The results here are well-known to the experts.

To apply this remark, we notice that $\pi \cong \bbZ[1/k]\rtimes \bbZ$, if $k\neq 0$, where  $aba^{-1}=b^k$ and the subgroup normally generated by $b$
$$ \la b, a^{-1}ba, a^{-2}ba^2, \dots\ra\subset \pi\ 
$$
is isomorphic to $\bbZ[1/k]$. 
The Wang exact sequence
$$ \dots \to H_i(\bbZ[1/k]; A) \xrightarrow{1-\alpha_*}  H_i(\bbZ[1/k]; A) \to
 H_i(\pi; A) \to  H_{i-1}(\bbZ[1/k]; A) \to \dots$$
 now gives another method of calculation for $A$ any $\Zpi$-module. Here
$\alpha\colon \bbZ[1/k] \to \bbZ[1/k]$ is the group automorphism ``multiplication by $ k$"  induced by conjugation by $a$. To see this, let $c_0 = b$, and define $c_i = a^{-i}ba^i$ for $i>0$. Then $\alpha(c_i) = c_i^k = c_{i-1}$ for $i>0$  and $\alpha(b) = b^k$. 
Since $\bbZ[1/k]$ is the direct limit $(\bbZ \xrightarrow{k} \bbZ \xrightarrow{k} \dots)$, with countable $\bbZ$-basis $\{c_0, c_1, \dots \}$, and homology commutes with direct limits, the maps $1-\alpha_*$ can be evaluated to arrive at Lemma~\ref{lem:homology2}. 
The same technique will be used for $L$-theory.

\begin{lemma}[{Waldhausen \cite[19.5]{waldhausen2}}] $\wh(B(k)) = 0$.
\end{lemma}
 Because of this result,  we can suppress $L$-theory torsion decorations and use $L \equiv L^h$ throughout the rest of this section.

We now compute the quadratic $L$-groups using a long exact Wang sequence \cite[p.~167]{cappell3},   \cite{ra3}:
$$\dots \to L_n(S) \xrightarrow{1-\alpha_*} L_n(S) \to L_n(\Zpi) 
\overset{\delta}{\to} L_{n-1}(S) \to \dots $$
where $S= \bbZ[\bbZ[1/k]]$.
There is a similar exact sequence for the symmetric $L$-groups (see
 \cite[4.1]{milgram-ranicki2}).

\begin{lemma}\label{inputb}
 For $\pi=B(k)$, we have
\begin{enumerate}
\item $L_4(\Zpi) \cong \bbZ \oplus \cy 2$ for $k$ odd.
\item  $L_4(\Zpi) \cong \bbZ$ for $k$ even.
\item $L_5(\Zpi)  \cong \bbZ \oplus \cy {(k-1)}$.
\item $L^0(\Zpi) \cong \bbZ$.
\end{enumerate}
\end{lemma}
\begin{proof}
Since the inclusion map $L_4(\bbZ) \to L_4(\bbZ[\bbZ])$ is an isomorphism, the map induced on $L_4(\bbZ[\bbZ])$ by  $\alpha$ is the identity (here $\bbZ[\bbZ]$ is the group ring generated by $\la b\ra$, and $\alpha(b) = b^k$). Since $L$-theory commutes with direct limits, it follows  that the inclusion $L_4(\bbZ) \to L_4(S)$ is an isomorphism. On the other hand, if $k$ is odd, the map induced by $\alpha$ on $L_3(\bbZ[\bbZ])$ is also the identity, so $L_3(\bbZ[\bbZ]) \cong L_3(S)\cong \cy 2$ detected by a codimension two Arf invariant.

 The group automorphism $\alpha\colon \bbZ[1/k] \to \bbZ[1/k]$ induces the identity map on both $L_3(S)$ and $L_4(S)$, since $\alpha$ acts as the identity on the coefficients $\bbZ$. By the  Wang sequence for the $L$-theory of a twisted Laurent ring, if $k$ is odd then 
 $$L_4(\Zpi) \cong L_4(\bbZ) \oplus L_2(\bbZ), $$
 detected by the ordinary signature and codimension two Arf invariant. This is a direct sum because $L_4(\bbZ)$ splits off, or alternatively, the map to $L_2(\bbZ)$ is induced by the boundary map $\delta$ in the Wang sequence above.
 
 The map induced by $\alpha$ on $L_5(\bbZ[\bbZ])$ is multiplication by $k$, so $L_5(S) = \bbZ[1/k]$. It follows that $\cok(1-\alpha_*\colon L_5(S) \to L_5(S)) = \cy{(k-1)}$, and hence $L_5(\Zpi) = \bbZ \oplus \cy{(k-1)}$ from the Wang sequence.

 Finally, we notice that $L^0(\Zpi) = L^2(\Zpi, -1)$ by the skew-suspension map (see \cite[Prop.~6.1]{ra10}). The long exact Wang sequence of \cite[4.1]{milgram-ranicki2}, together with similar calculations of the direct limits shows that $L^0(\bbZ) \cong L^0(\Zpi)$.
\end{proof}

\begin{definition} \label{def:Arf}
The \emph{codimension two Arf invariant}  for $\pi=B(k)$, $k$ odd, is the projection
$$\Arf\colon L_4(\Zpi) \to  \widetilde L_4(\Zpi) = L_4(\Zpi) / L_4(\bbZ) \cong   L_2(\bbZ)=\cy 2$$
 It is defined for any non-singular even form $(V, \lambda_V)$ on a finitely-generated free $\La$-module, as the Arf invariant of the element $[\lambda_V] \in L_4(\Zpi)$. If $M$ has an even equivariant intersection form $s_M$, we define $\Arf(M) := \Arf(\sMdag)$.
 \end{definition}

\begin{remark} We can compare the Wang sequences for computing $H_*(\pi;\bbZ)$ or $H_*(\pi;\cy 2)$ and the Wang sequence for the $L$-groups $L_*(\Zpi)$ via the universal homomorphisms in the assembly map
$$A_*\colon H_*( \Bpi; \bbL_0(\bbZ)) \to L_*(\Zpi)$$
as described in \cite[\S 1]{hmtw1}.  
In our case,  we only need the homomorphisms
$$\cI_j \colon H_j(\pi; \bbZ) \to L_j(\Zpi), \text{\ for \ } j=0,1$$
and
$$ \kappa_j\colon
H_j(\pi;\cy 2) \to L_{j+2}(\Zpi), \text{\ for \ } j = 0, 1, 2$$
In general
the range of these homomorphisms should be localized at 2, but in low dimensions we have integral lifts for these maps (see Kirby-Taylor \cite[\S 2-3]{kirby-taylor1}).
\end{remark}
These homomorphisms give natural transformations between the two Wang sequences.

\begin{lemma} 
The assembly map $A_*$ is an isomorphism for $* = 4,5$.
\end{lemma}
\begin{proof} The domains of the assembly maps are given by
 $$H_4( \Bpi; \bbL_0(\bbZ)) = H_0(\pi;\bbZ) \oplus H_2(\pi;\cy 2)$$ and $$H_5( \Bpi; \bbL_0(\bbZ)) = H_1(\pi, \bbZ)\ .$$
 The result follows by naturality of the Wang sequences for homology and $L$-theory, and the calculations  in Lemma \ref{lem:homology2} and Lemma \ref{inputb}.
 \end{proof}
It follows that the surgery obstruction groups in these dimensions are generated by closed manifold surgery obstructions. Explicitly, we use
$S^1\times E8$ with two  injections of $\bbZ \to \pi$, distinct up to conjugation.

\begin{proof}[The proof of Theorem A] 
We have now shown that the groups $B(k)$ satisfy properties (W-AA) in Section~\ref{sec:BS}. Since $H^2(\pi;\cy 2) \subseteq \cy 2$ for the solvable Baumslag-Solitar groups, the $w_2$-type is equivalent to the type and hence Theorem~A follows from Theorem~C. 
\end{proof}
For manifolds $M$ with type (III) there is  a relation between the Kirby-Siebenmann invariant and the codimension two Arf invariant from Definition~\ref{def:Arf}.  We first give a more general result.

 \begin{theorem}\label{thm:Arf}
Assume that $M$ is a closed oriented $4$-manifold with $s_M$ even and $w_2$-type $(\pi,w)$, that lies in the $F_{2,2}$-term of the filtration for the James spectral sequence of the normal $1$-type. Then one has
$$
KS(M)\equiv {\frac{\sign(M)}{8}} + \langle w, \rho(M,\tilde\nu) \rangle \ (\Mod 2),
$$
where $\rho\colon F_{2,2}\, \Omega_4(\pi,w) \twoheadrightarrow E_{2,2}^\infty \cong H_2(\pi;\cy 2)/\Image(d_2, d_3)$ is the  natural projection and $\tilde\nu$ is a normal 1-smoothing for $M$. In particular, $w$ evaluates trivially on the images of the differentials ending in $E_{2,2}^r$.
\end{theorem}
\begin{proof} 
As in Lemma~\ref{stable_bordism}, the normal 1-type $B(\pi,w)$ pulls back from $BSTOP$ via the map $w\colon K(\pi,1) \to K(\cy 2,2)$. We compare the James spectral sequences for these two fibrations, knowing that oriented topological bordism is classified by
 \[
\Omega_4(*) \cong \bbZ \oplus \cy 2
\]
 via the signature and Kirby-Siebenmann invariant. From the well known computations
 \[
 H_4(K(\cy 2,2);\bbZ) \cong\cy 4  \quad \text{ and } \quad H_i((K(\cy 2,2);\cy 2) \cong\cy 2, \ i=2,3,
 \]
it follows that the $F_{2,2}$-term in this case is $8\bbZ\oplus \cy 2$. Moreover, the quotient $\cy 8$ must be the signature modulo  $8$. Looking at the examples of $\CP^2$, $\CH =*\CP^2$ and $E8$, we see that the Kirby-Siebenmann invariant on the $F_{2,2}$-term is given by 
\[
\sign/8 \, (\Mod 2)+ p_2\colon \Omega_4(*)\to \cy 2
\]
where $p_2$ is the projection onto the second $\cy 2$. On comparing this to $ \Omega_4(\pi,w)$, the only remaining observation is that $w\colon K(\pi,1) \to K(\cy 2,2)$ induces a map 
\[
w_* : H_2(\pi;\cy 2)/\Image(d_2, d_3) \to H_2(K(\cy 2,2);\cy 2) \cong\cy 2
\]
that translates into the evaluation $ \langle w, - \rangle $ used in the statement of the Theorem.
\end{proof}
\begin{corollary}\label{cor:Arf}
If $M$ is a closed oriented $4$-manifold of type \textup{(III)}, with solvable Baumslag-Solitar fundamental group, then
$$KS(M)\equiv {\frac{\sign(M)}{8}} + \Arf(M)\ (\Mod 2).$$
\end{corollary}
\begin{proof} 
 For $\pi \cong B(k)$ the group $H_2(\pi;\cy 2)$ is either $0$ or $\cy 2$, depending on whether $k$ is even or odd. For type~(III) we must be in the latter case and hence our $w_2$-type is $(\pi,w)$ with $w\neq 0$. By Lemma~\ref{stable_bordism} we have an isomorphism
 \[
\Omega_4(\pi,w) \cong 8\bbZ\oplus H_2(\pi;\cy 2)  \cong 8\bbZ\oplus \cy 2,
\]
where the isomorphism on the right hand side is induced by evaluating $w$. 
Both sides of the equation define a homomorphism on this group and we want to prove equality. 
By Lemma~\ref{lem:theta_relation}, the Arf-invariant gives the nontrivial projection onto the second summand $\cy 2$. Thus our result follows from Theorem \ref{thm:Arf}.
\end{proof}
 
 \begin{proof}[Proof of Theorem B]
 The first statement of Theorem~B  has already been proved in Corollary~\ref{cor:rad} and Corollary~\ref{cor:free}.  It remains to establish parts (i) and (ii) concerning the realizability of the forms. Note that by topological surgery  \cite{freedman1}, a form is realizable if and only if it is stably realizable, where ``stably" means after orthogonal sum with hyperbolic forms or after connected sum with $S^2\times S^2$'s for forms and manifolds respectively.

Suppose that $(F, \lambda)$ is a non-singular hermitian form on a finitely-generated, stably-free $\La$-module $F$. After stabilizing by hyperbolic forms, we may assume that $F$ is $\La$-free. If $F$ is an odd form, then it represents an element in  the Witt group $L^0(\Zpi) = \bbZ$, so there exists metabolic forms $\omega$, $\omega'$ on free $\La$-modules so that $\lambda \perp \omega \cong \lambda_0 \perp \omega'$, where $\lambda_0$ is a standard form (diagonal $\pm 1$) with the same signature as $\lambda$. By \cite[Lemma 3]{hteichner1}, we may assume that $\omega$ and $\omega'$ are hyperbolic forms, and hence $\lambda$ is stably realizable. Since the form $\lambda_0$ can be realized by manifolds with different $KS$-invariants, there are two possibilities  (as stated in part (ii)).

Now suppose that $(F, \lambda)$ is an even form. Then there exists a quadratic refinement $\mu$ so that $(F, \lambda, \mu)$ represents an element of
$L_4(\Zpi) = \bbZ \oplus \cy 2$ ($k$ odd) or $L_4(\Zpi) = \bbZ$ ($k$ even).  By Lemma \ref{stable_bordism}, any element of the $L$-group is realizable by a manifold of type II or III. It follows that  $(F, \lambda)$ is stably realizable, and hence realizable by a $4$-manifold.  For the remaining assertion in  part (ii), note that 
 the Kirby-Siebenmann invariant is determined by the intersection form (see Corollary \ref{cor:Arf} for $k$ odd).
\end{proof}
\begin{remark} For manifolds with geometrically $2$-dimensional fundamental group $\pi$, the same argument proves that any such even form  is realizable by an $s$-cobordism, whenever the assembly map $A_4\colon H_4(\pi; \bbL_0) \to L_4(\Zpi)$ is an isomorphism.
\end{remark}

\providecommand{\bysame}{\leavevmode\hbox to3em{\hrulefill}\thinspace}
\providecommand{\MR}{\relax\ifhmode\unskip\space\fi MR }
\providecommand{\MRhref}[2]{%
  \href{http://www.ams.org/mathscinet-getitem?mr=#1}{#2}
}
\providecommand{\href}[2]{#2}

\end{document}